\documentclass[11pt,letterpaper]{article}
\usepackage[margin=1in]{geometry}

\usepackage{physics}
\usepackage{mathtools}
\usepackage{color}
\usepackage{amssymb}
\usepackage{amsthm}
\usepackage[nothing]{algorithm}
\usepackage{algorithmic}
     
\usepackage{amsmath}
\usepackage[pdftex]{graphicx}
\usepackage{hyperref}

\newcommand{\sfT}{\top}
\newcommand{\lp}{\left(}
\newcommand{\rp}{\right)}
\newcommand{\bbS}{\mathbb{S}}
\def\bal{\begin{aligned}}
\def\eal{\end{aligned}}

\newcommand{\mca}{\mathcal{A}}
\newcommand{\mcb}{\mathcal{B}}

\newcommand{\bo}{B_{\mathrm{old}}}
\newcommand{\bn}{B_{\mathrm{new}}}

\newcommand{\an}{A_{\mathrm{new}}}

\usepackage{color}

\newtheorem{theorem}{Theorem}[]
\newtheorem{lemma}[theorem]{Lemma}

\theoremstyle{definition}

\theoremstyle{remark}

\newcommand{\R}{\mathbb{R}}

\newcommand{\be}{\mathbf{e}}

\newcommand{\psd}{\mathbb{S}_+}
\newcommand{\la}{\langle}
\newcommand{\ra}{\rangle}

\newcommand{\lam}{\lambda}

\newcommand{\amin}{\mathrm{argmin}}

\DeclarePairedDelimiterX{\inn}[2]{\langle}{\rangle}{#1, #2}

\DeclareMathOperator{\new}{{\rm new}}
\DeclareMathOperator{\old}{{\rm old}}

\DeclareMathOperator{\cala}{\mathcal{A}}

\def\be{\begin{equation}}
\def\ee{\end{equation}}
\newcommand{\bx}{\mathbf{x}}

\title{A Non-commutative Extension of  Lee-Seung's Algorithm for Positive Semidefinite Factorizations}

\author{
	Yong Sheng Soh \thanks{Y.~S.~Soh (email: \url{matsys@nus.edu.sg}) is with the Department of Mathematics, National University of Singapore and the Institute of High Performance Computing, Agency for Science, Technology and Research.}
	\and Antonios Varvitsiotis\thanks{A.~Varvitsiotis (email: \url{avarvits@gmail.com}) is with the Engineering Systems and Design Pillar, Singapore University of Technology and Design.}}

\begin{document}

\maketitle

\begin{abstract}
Given a data matrix $X\in \mathbb{R}_+^{m\times n}$ with non-negative entries, a Positive Semidefinite (PSD) factorization of $X$ is a collection of $r \times r$-dimensional PSD matrices $\{A_i\}$ and $\{B_j\}$ satisfying the condition $X_{ij}= \mathrm{tr}(A_i B_j)$ for all $\ i\in [m],\ j\in [n]$.  PSD factorizations are fundamentally linked to understanding the expressiveness of semidefinite programs as well as the power and limitations of quantum resources in information theory.  The PSD factorization task generalizes the Non-negative Matrix Factorization (NMF) problem in which we seek a collection of $r$-dimensional non-negative vectors $\{a_i\}$ and $\{b_j\}$ satisfying $X_{ij}= a_i^T b_j$,  for all $i\in [m],\ j\in [n]$ -- one can recover the latter problem by choosing matrices in the PSD factorization to be diagonal.  The most widely used algorithm for computing NMFs of a matrix is the Multiplicative Update algorithm developed by Lee and Seung, in which non-negativity of the updates is preserved by scaling with positive diagonal matrices.  In this paper, we describe a non-commutative extension of Lee-Seung's algorithm, which we call the Matrix Multiplicative Update (MMU) algorithm, for computing PSD factorizations.  The MMU algorithm ensures that updates remain PSD by congruence scaling with the matrix geometric mean of appropriate PSD matrices, and it retains the simplicity of implementation that the multiplicative update algorithm for NMF enjoys.  Building on the Majorization-Minimization framework, we show that under our update scheme the squared loss objective is non-increasing and fixed points correspond to critical points.  The analysis relies on  Lieb's Concavity Theorem.  Beyond PSD factorizations, we show that the MMU algorithm can be also used as a primitive to calculate block-diagonal PSD factorizations and tensor PSD factorizations.  We demonstrate the utility of our method with experiments on real and synthetic data. 
\end{abstract}

\section{Introduction}

Let $X\in \R_+^{m\times n}$ be a $m \times n$ dimensional matrix with non-negative entries and  $r\in~\mathbb{N}$ a user-specified parameter.  
An {\em $r$-dimensional  positive semidefinite (PSD) factorization} of $X$ is given  by two families  of 
  $r\times r$  PSD matrices $A_1,\ldots, A_m$ and $B_1,\ldots,B_n$  satisfying 
\be\label{factorization}
X_{ij}= \tr(A_iB_j),\   i\in [m],\ j\in [n].
\ee
%
Every non-negative matrix admits an $r$-dimensional PSD factorization for an appropriate value of~$r\in\mathbb{N}$--we may, for instance, take $A_i={\rm diag}(X_{i:})$ and $B_j={\rm diag}(e_j)$, a choice that corresponds to  an $n$-dimensional PSD factorization.  The smallest $r\in \mathbb{N}$ for which $X$ admits an $r$-dimensional  PSD factorization is called   the \emph{PSD-rank} \cite{psdrank}.


 PSD factorizations are of fundamental importance   to a wide  range of areas, most notably
 towards  understanding  the  expressive power  of linear optimization over the cone of  positive semidefinite matrices,   \cite{lifts, sam}, 
 studying  the power and  limitations of quantum resources within the framework  of information theory \cite{qcorr,sam}, and  as a natural non-commutative generalization of the  extremely popular dimensionality reduction technique  of    non-negative matrix factorizations (NMFs) \cite{PT, Lsnature, LS00}. We elaborate further on the relevance of PSD factorizations to each of these areas~below.

%
 
\paragraph{Links to Semidefinite Programming.}  A Semidefinite Program (SDP) is a convex optimization problem  in which we minimize a linear function over the set of PSD matrices intersected with an affine subspace.  SDPs are a powerful generalization of Linear Programs
with extensive modeling power and tractable algorithms for solving them, e.g., see \cite{sdp} and references therein.  SDPs are frequently used as convex relaxations to combinatorial problems, and have many important applications   including   optimal power flow computation  \cite{lavaei}, robustness certification to adversarial examples in neural networks \cite{adversarial}, and inference in graphical models \cite{inference}. 

 Given a bounded  polytope $P = \{x\in \R^d: c_i^\sfT x\le d_i, \ i\in [\ell]\}={\rm conv}(v_1,..,v_k)$,  a basic question concerning  the expressive power of SDPs is  to find  the \emph{smallest} possible SDP description of $P$, i.e.,  
 the minimum $r\in \mathbb{N}$ for which we can express $P$ as the projection of an the affine slice %
 of the cone of $r\times r$  PSD matrices. Concretely, the goal in this setting  is  to express $P$ as:
\be\label{psdlift}
P=\pi(\psd^r \cap \mathcal{L}),
\ee
where $\psd^r $ is the cone of  $r\times r$  PSD matrices, $ \mathcal{L}$ is an affine subspace of the space of $r\times r$ symmetric matrices and $\pi$ a linear projection from the space of $r\times r$ symmetric matrices to $\R^d$.   A representation  of the form  \eqref{psdlift} is called an {\em extended formulation} or  {\em PSD-lift} of $P$ and     is extremely useful for optimization purposes. Indeed,  the existence of a PSD-lift immediately implies   that $\min \{\la c, x\ra : x\in P\}=\min\{ \la \pi^\sfT(c), y\ra: y\in  \psd^r \cap \mathcal{L}\}$, and consequently,  linear optimization over $P$  (which can be hard) corresponds to an SDP (which can be solved efficiently).

 
 To  describe the connection of PSD factorizations with  SDP-lifts, 
let  $S_P$ be the  {\em slack matrix} of $P$,  namely, $S_P$ is a rectangular matrix whose rows are indexed by the facets of $P$, its columns indexed by  extreme points $v_j$,  and the $ij$-entry of $S_P$ corresponds to the slack between the $i$-th facet and the  $j$-th vertex,    $S_{ij}=d_i- c_i^\sfT v_j$.    Generalizing a seminal result by Yannakakis for LPs~\cite{yannakakis}, it was shown independently in \cite{lifts} and \cite{sam} that if $S_P$ admits an $r$-dimensional PSD factorization, the polytope $P$ 
admits a PSD-lift over the cone of $r\times r$  PSD matrices. The proof is also constructive -- given a PSD factorization of $S$, there is an explicit description of $\mathcal{L}$ and $\pi$ that gives rise to $P$.

An important special case of the PSD factorization problem is when the PSD factors are block-diagonal PSD matrices, where both the number of  blocks and the size of each block is fixed,~i.e., 
$A_i,B_j\in (\psd^r)^k.$
For a fixed and user-specified $r\in \mathbb{N}$, the least $k\in \mathbb{N}$ for which $X\in \R_+^{m\times n}$   admits a PSD factorization with PSD-factors in $(\psd^r)^k$ is called the $r$-block diagonal PSD-rank  of $X$. In terms of the geometric interpretation of PSD factorizations, block-diagonal PSD factorizations of the slack matrix $S_P$ 
 correspond to  extended formulations of $P$ over  a Cartesian product of  PSD cones, i.e.,  $P=\pi((\psd^r)^k \cap \mathcal{L}).$
In terms of relevance to optimization,  extended formulations over~$(\psd^r)^k$    allow to perform linear optimization over $P$ by solving    block-diagonal SDPs, which can be solved numerically much faster compared to dense SDPs. 
In fact,  most  interior-point algorithms for SDPs are  designed to exploit block-diagonal structure if it present in the problem.  The first systematic study of block-diagonal PSD-lifts was given in \cite{FP13}, where the focus was mainly on lower~bounds.

\paragraph{Links to quantum information theory.}
Consider two  parties, Alice and Bob,   that try to generate samples $(i,j)$ following some joint distribution $P(i,j)$. For this section, it  is crucial  to   think  of the distribution  $P(i,j)$ as being  arranged in an entrywise  non-negative matrix, and we use $P$ to  simultaneously refer  to both the  distribution and its  matrix  representation.  Clearly, if $P$ is not  a product distribution, Alice and Bob should   either communicate  or share some common information to be able  to generate samples according to $P$. In the {\em correlation generation problem}, the goal is to find  the least amount of shared resources that are needed to achieve this task. The considered resources can be  either  classical (shared randomness), quantum (shared entangled state) or hybrid. 

In the quantum case,   correlation generation  boils down to finding  a quantum state   $\rho \in \psd^{r^2}$ with $\tr(\rho)=1$ and quantum  measurements  $E_i, F_j$ (i.e., $E_i,F_j\in \psd^r$ and $\sum_iE_i=\sum_jF_j=I$) such~that 
 \be\label{quantum}
 P(i,j)=\tr( (E_i\otimes F_j)\rho), \   i\in [m],\ j\in [n].
 \ee 
 The least $r\in \mathbb{N}$ for which a factorization of the form \eqref{quantum} is possible is 
 given by the (logarithm) of  the  psd-rank of the matrix $P$~\cite{qcorr}. Moreover, the proof of \cite{qcorr} is constructive, in the sense that, given a $r$-dimensional  PSD factorization of $P$, there is an  explicit description of a quantum state $\rho \in \psd^{r^2}$ and  measurement operators acting on $\mathbb{C}^r$ that satisfy \eqref{quantum}, e.g., see \cite[Proposition 3.8]{psdrank}. 


Moving beyond purely quantum protocols, there has been  recent interest   in  hybrid classical-quantum  protocols, motivated by the fact that near-term quantum devices can  only operate reliably on a limited number of qubits  \cite{blockpsd}. Specifically, assuming that their  quantum capabilities are limited to manipulating $s$ qubits,  hybrid classical-quantum protocols that allow    to generate samples  from a joint distribution $P(i,j)$, correspond to PSD factorizations of~$P$ where the PSD factors are block-diagonal, with block-size   at most $2^s$. Moreover, the minimum amount of classical resources required in a classical-quatum protocol with $s$ qubits, is given by the $2^s$-block diagonal PSD-rank  of $P$. 


\paragraph{Links to nonnegative matrix factorizations.}
 An $r$-dimensional {\em nonnegative matrix  factorization} (NMF) of  $X\in \R_+^{m\times n}$ \cite{PT,Lsnature} is specified by two families  of $r$-dimensional entrywise nonnegative vectors  $a_1,\ldots, a_m\in \R^r_+$ and $b_1,\ldots,b_n\in \R^r_+$  satisfying 
\be\label{nmf}
X_{ij}= \la a_i,b_j\ra,\   i\in [m],\ j\in [n].
\ee
NMF is  a widely used  dimensionality reduction tool that
%
 gives  {\em parts-based representation} of the input data, as   it only allows for additive, not subtractive, combinations.
To make this point clear, note that an equivalent reformulation of an $r$-dimensional  NMF   \eqref{nmf} is   $X=AB$
where $A\in \R^{m\times r}_{+}$ is the matrix whose rows are the $a_i$'s and the matrix $B\in \R_{+}^{r\times n}$ has as columns  the $b_j$'s, or equivalently,
\be\label{nmf2}
X_{:j}\in {\rm cone}(A_{:1}, \ldots, A_{:r}), \ j \in [n].\ee
 The equivalent viewpoint   for NMFs given in \eqref{nmf2} is  more amenable to interpretation, as it  gives a representation of  each column of $X$ (i.e., each data point)  as nonnegative (and thus additive)   combination of the $r$ columns of $A$, and the columns of $B$ give the coefficients of the conic~combination.  
NMF factorizations   have  applications  in many areas,  notable examples including  document clustering \cite{XLG03}, music analysis \cite{fevotte},     speech-source separation  \cite{speech-sep} and  cancer-class identification~\cite{GC05}. For a comprehensive discussion  on  NMFs  the reader is referred to  the survey \cite{gilis} and references~therein.

NMF factorizations are a special case of PSD factorizations where the $r\times r$ PSD matrices $A_i$ and $B_j$ are {\em diagonal}, i.e.,  we have that  $A_i={\rm diag}(a_i)$ and $ B_j={\rm diag}(b_j)$ for some vectors $a_i,b_j\in\R^r_+$ (recall that a diagonal matrix is PSD iff its diagonal entries are nonnegative). Moreover,  given a PSD factorization of $X$, (i.e., $X_{ij}=\tr(A_iB_j)$) for which all the PSD factors $A_i,B_j$ commute (and thus can be simultaneously diagonalized), corresponds to an NMF factorization.  
In this sense, PSD factorizations  are  a natural non-commutative generalization of  NMF factorizations.
 
 \paragraph{Interpretability of PSD factorizations.} An equivalent way to define   an {$r$-dimensional} NMF for a data matrix $X$ (cf. \eqref{nmf})  is through the  existence of  a liner mapping $\mathcal{A} : \mathbb{R}^{r} \rightarrow \mathbb{R}^{m}$  satisfying
 \be\label{nmflatent}
 X_{:j}\in \mathcal{A}(\R^r_+) \ \text{ for all }  j\in [n] \quad \text{ and } \quad  \cala(\R^r_+)\subseteq \R^n_+. 
 \ee
Consequently, the mapping $\cala$ (or rather, 
the  image of the extreme rays of the cone $\R^r_+$ under $\mathcal{A}$),  describe a latent space 
that can generate   all data points  $X_{:j}$ via nonnegative combinations. 

Analogously, in the setting of PSD factorizations, the existence of an $r$-dimensional PSD factorization of  $X$ (cf.~\eqref{factorization})  is equivalent to the existence of a linear mapping $\cala: \mathbb{S}^r \to \R^m$ satisfying
 \be\label{psdlatent}
 X_{:j}\in \mathcal{A}(\psd^r) \ \text{ for all }  j\in [n] \quad \text{ and } \quad  \cala(\psd^r)\subseteq \R^n_+. 
 \ee
Comparing  \eqref{nmflatent} and \eqref{psdlatent}, the difference between NMF and PSD factorizations is immediately apparent. In the setting of PSD factorizations the latent pace is {\em infinite-dimensional}, and specifically, it is the  image of the extreme rays of the cone of $r\times r$ PSD matrices (i.e., all matrices $uu^\sfT$ where $u\in \R^r$) under~$\mathcal{A}$. In this  latent space, each data  point $X_{:j}$ is represented by a PSD matrix $B_j\in \psd^r$, and using 
its spectral decomposition  $B_j=\sum_{i=1}^r\lambda_iu_iu_i^\sfT$, leads to the representation
$X_{:j}=\sum_i\lambda_i\cala(u_iu_i^\sfT).$ Additional details and explicit  examples demonstrating the qualitative difference in expressive power  between  NMF and PSD factorizations are given in Section \ref{sec:numerical}. 

\section{Prior Works on PSD Factorizations and Summary of Results}
 A canonical  starting point for finding an (approximate) $r$-dimensional PSD factorization of a given  matrix $X\in \R_+^{m\times n}$  is     to solve the non-convex optimization problem
\be\label{bigboss}
\inf  \sum_{i,j}(X_{ij}-\tr(A_iB_j))^2 \quad \text{s.t.}\quad  \ A_1, \ldots, A_m, B_1, \ldots B_n\in \psd^r,
\ee
aiming to  find  an approximate $r$-dimensional PSD factorization that minimizes   the square loss over all entries of $X$. 
Fixing  one of the two families  of matrix variables, say   the $A_i$'s, problem \eqref{bigboss} is  separable with respect to  $B_1,\ldots, B_m$. Consequently,   a reasonable solution approach for \eqref{bigboss}  is to  alternate between updating the $A_i$'s and $B_j$'s by solving the   sub-problems:
\begin{align}
A_i&\leftarrow  \arg \inf~~ \sum_{i,j}(X_{ij}-\tr(A_iB_j))^2 \quad \text{s.t.} \quad  A_1,\ldots, A_m\in \psd^r  \label{alt1}\\
B_j&\leftarrow \arg \inf ~~ \sum_{i,j}(X_{ij}-\tr(A_iB_j))^2 \quad \text{s.t.} \quad   B_1,\ldots,B_n\in \psd^r  \label{alt2}
\end{align}
The two sub-problems   in each update step are symmetric in the variables $A_i$ and $B_j$, with the small modification where we replace $X$ with its transpose.  As such, for the remainder of this discussion, we only focus on the sub-problem \eqref{alt2} corresponding to fixing the $A_i$'s and updating the~$B_j$'s. 
 Moreover, \eqref{alt2} is separable with respect to each variable $B_i$, so it suffices to focus on 
 \be\label{acfsdv}
{\inf}  \sum_{i}(X_{ij}-\tr(A_iB_j))^2 \quad \text{s.t.} \quad  B_j\in \psd^r.
 \ee
Lastly, to simplify notation we omit subscripts, and specifically, we denote  by  $x$  the $j$-th column of $X$ and by $B$ the PSD matrix variable~$B_j$.   Defining  $\mathcal{A} : \mathbb{S}^{r} \rightarrow \mathbb{R}^{m}$ to be the linear map 
$	\mathcal{A}(Z) = \left(
	\langle A_1,Z \rangle, \ \ldots \ 
	,\langle A_m,Z \rangle
	 \right),
$  problem~\eqref{acfsdv} can be then equivalently written~as
 \be\label{main}
 { \inf} ~~ \| x - \mathcal{A} (B) \|_2^2 \quad \mathrm{s.t.} \quad  B \in \psd^r.
 \ee
 
 The optimization problem \eqref{main} is convex, and in fact, falls within  the well-studied class of convex quadratic SDPs. Nevertheless, there is no closed-form solution for this family  of  optimization problems, and consequently,   typical solution strategies  rely on     numerical optimization,  e.g., see~\cite{toh}.

 \paragraph{Summary of results.}    In this paper we introduce  and study an iterative algorithm  (Algorithm~\ref{alg}) we call the {\em Matrix Multiplicative Update} (MMU) algorithm for computing PSD factorizations. The MMU algorithm builds  on the Majorization-Minimization framework, and as discussed in the previous section, the main workhorse is an   iterative algorithm for  the convex quadratic SDP~\eqref{main}.  

From a computational perspective, the iterates of the MMU   algorithm  are updated via  conjugation with appropriately defined matrices,  so our method  has the advantage of being {\em simple to implement} and moreover,  the PSDness of the iterates is {\em automatically guaranteed}. From a theoretical perspective, the squared loss objective is non-increasing along the algorithms' trajectories (Theorem \ref{LSstylethm}) and moreover,  its fixed points satisfy the  first-order optimality  conditions (Theorem~\ref{fixedpointsthm}).  The analysis of the MMU algorithm  relies on the use of several operator trace inequalities (including Von Neumann's trace inequality and Lieb's Concavity Theorem).

An important feature of the MMU  algorithm is that if it is initialized with block-diagonal PSD matrices,  the same block-diagonal structure is {\em preserved throughout  its execution}, which 
    leads to  an algorithm  for calculating block-diagonal PSD factorizations. 
  In particular, in Section \ref{sec:applications} we show that if the MMU  algorithm is initialized with  {\em diagonal PSD matrices}, the iterates remain diagonal PSD throughout,  and as it turns out, our algorithm in this case  reduces to   Lee-Seung's seminal Multiplicative Update algorithm  for   computing NMFs \cite{LS00}. Moreover, we show how to the MMU algorithm can be used as a primitive to calculate PSD factorizations of nonnegative tensors. In terms of numerical experiments,  in Section \ref{sec:numerical} we demonstrate the utility of our method for  both synthetic (random Euclidean Distance matrices) and real data (CBCL image dataset). 
   
 \paragraph{Existing work.} 
    
All 
existing algorithms for computing PSD factorizations  employ the alternating minimization approach described in the previous section, where we fix one set of variables and minimize over the other, and essentially boil  down into finding algorithms for  the convex problem~\eqref{main}.


\medskip 
\noindent {\em Projected Gradient Method (PGM).}  The first approach for computing PSD factorizations  is based on  applying PGM to \eqref{main}, alternating  between a gradient step to minimize the objective and  a projection step onto the set of PSD matrices \cite{VGG18}.  The latter projection step uses the following useful fact: Given the spectral decomposition $C=U{\rm diag}(\lam_i)U^\sfT$ of a matrix $C\in \bbS^n$, the projection onto the PSD cone is $U{\rm diag}(\max(0,\lam_i))U^\sfT$ \cite{hig}.  The vanilla PGM has slow convergence rate, so the authors in \cite{VGG18} also propose an accelerated variant that incorporates a momentum term.

\medskip 

\noindent {\em Coordinate Descent.}  The authors in \cite{VGG18} also propose a different algorithm combining the ideas of coordinate descent and  a change of variables that allows them to also control  the rank of the PSD factors, which  was popularized by  the seminal work of Burer and Monteiro for solving rank-constrained SDPs \cite{BM}. 
  Concretely, the authors use the  parameterization  $A_i=a_ia_i^\sfT$, and $B_j=b_jb_j^\sfT$, where $a_i\in \R^{r\times r_{A_i}}$, and $b_j\in \R^{r\times r_{B_j}}$ for some fixed $r_{A_i}, r_{B_j}\in \mathbb{N}$, and optimize using a coordinate descent scheme over  the entries of the matrices $a_i$ and $b_j$.
In this setting,    problem \eqref{main}  
 is a quartic polynomial in the entries of $b$.  Thus, its gradient is a cubic polynomial, 
and its roots can be found  using  Cardano's method (and careful book-keeping). 
 

\medskip 
\noindent {\em Connections to Affine Rank Minimization and Phase Retrieval.}  A different set of algorithms developed in \cite{LF20, LF20b, LLTF20} is based on the connections between computing PSD factorizations with the affine rank minimization (ARM) and the phase retrieval (PR) in signal procesing.  
First, recall that the PSD-ARM problem focuses on  recovering a low-rank matrix from affine measurements:
$$
 \min ~~ \rank(B) \quad \text{s.t.} \quad   \ \cala(B)=x, \ B\in \psd^r.
$$
Here, $\cala$ is a known linear map representing measurements while $x$ is known vector of observations. Due to the non-convexity of the rank function, a useful heuristic initially popularized in the control community  is to replace the rank by  the trace function, e.g., see \cite{MP} and \cite{fazel}, in which case the resulting problem is an instance of an SDP.  A different  heuristic for   PSD-ARM is to find a PSD matrix of rank at most $k$ that  minimizes the squared loss function, i.e., \be\label{cdsfvsdfv}
 \inf~~ \|x-\cala(B)\|_2^2 \quad \text{s.t.} \quad  B\in \psd^r, \ \rank(B)\le k,
\ee
where alternatively, the rank constraint can be enforced  by parametrizing the PSD matrix variable $B\in \psd^r$ as $B=bb^\sfT$ with $b\in \R^{r\times k}$. 
The point of departure for the works  \cite{LF20, LLTF20, LF20b} is that problem~\eqref{cdsfvsdfv}  corresponds exactly to  the sub-problem \eqref{main} encountered in any alternate minimization strategy for computing  PSD factorizations,   albeit with an additional  rank constraint. In view of this, any  algorithm  from the signal processing literature developed for ARM can be applied to \eqref{main}.


The main algorithms considered in \cite{LF20, LF20b, LLTF20}  are Singular Value Projection (SVP)  \cite{SVP},   Procrustes Flow \cite{procrsustes}, and variants thereof.
In terms of convergence guarantees,  for affine maps $\mathcal{A}$  obeying the  Restricted Isometry Property \cite{RIP}, both algorithms converge to an optimal solution. Nevertheless, it is unclear whether these guarantees carry over when applied to the PSD factorization problem. 
 
\paragraph{Roadmap.} In Section \ref{sec2} we derive our MMU algorithm for computing PSD factorizations and in Section \ref{sec:fp} we show that its fixed points   correspond to KKT points. In Section \ref{sec:applications} we give various theoretical applications of the MMU algorithm and in Section \ref{sec:numerical} we go from theory to practise and apply the MMU algorithm to synthetic and real datasets.

\section{A Matrix Multiplicative Update  Algorithm for  PSD Factorizations}\label{sec2}

In this section we describe our algorithm for computing (approximate) PSD factorizations of a matrix $X$. As already discussed, we employ an alternate optimization approach where we alternate between optimizing with respect to the $A_i$'s and the $B_j$'s. 
The sub-problem in each update step symmetric in the variables $A_i$ and $B_j$, with the small modification where we replace $X$ with its transpose.  As such, in the remainder of this discussion, we assume without loss of generality that   the $A_i$'s are fixed and we update the $B_j$'s. The resulting sub-problem we need to solve is \eqref{main}.


\paragraph{Majorization-Minimization (MM) Framework.}  Our algorithm is an instance of the (MM) framework, e.g. see \cite{Lange} and references therein. To briefly describe this approach, suppose   we need to solve the optimization problem  $\min \{F(x): x\in \mathcal{X}\}.$  
The MM framework  relies on the existence of a parametrized family of  auxilliary functions  $u_x: \mathcal{X}\to \R$, one for each $x\in \mathcal{X}$,~where:
\be\label{auxilliary}
 F(y)\le u_x(y),  \ \text{for all } y\in \mathcal{X}\  \text{ and } \ F(x)=u_x(x).
 \ee
Based on these two properties,  $F$ is nonincresasing under the  update rule:
 \be\label{updaterule}
  x^{\new} =\amin \{ u_{x^{\old}}(y):\ y \in \mathcal{X}\},
  \ee
as can be easily seen by the following chain of inequalities
$$f(x^{\new})\le u_{x^{\old}}(x^{\new})\le u_{x^{\old}}(x^{\old})=f(x^{\old}).$$

We conclude with two important remarks concerning  the MM framework. First, note that although the iterates generated by the MM  update rule  \eqref{updaterule} are nonincreasing  in objective function value, there is in general no guarantee that they converge to a minimizer. Secondly,  for the MM approach to be of any use, the auxilliary functions employed at each iteration need to be easy to~optimize.


\paragraph{Matrix Geometric Mean.}  Our choice of auxilliary functions relies  on the well-studied  notion of a geometric mean between a pair  of positive definite matrices, whose definition  we recall next. For additional details and omitted proofs  the reader is referred to  \cite{gmeam} and \cite{bhatia}.
  The matrix geometric mean of  two positive definite matrices $C$ and $D$ is given by 
\begin{equation}\label{GM}
	C\# D= C^{1/2} (C^{-1/2} D C^{-1/2})^{1/2} C^{1/2},
\end{equation}
or equivalently, it is the unique positive definite solution of the Riccati equation 
\be\label{riccati}
XC^{-1}X=D,
\ee in the matrix variable $X$. The matrix geometric mean also has a nice geometric interpretation in terms of  the Riemannian geometry     of the manifold of  positive definite matrices, and specifically, $C\# D$   is the midpoint of the unique geodesic joining $C$ and $D$. 
Finally, the  matrix geometric mean is symmetric in its two arguments  $C \# D = D \# C$ and also satisfies $(C\#D)^{-1} = C^{-1} \# D^{-1}$. 

\paragraph{The MMU Algorithm for PSD Factorizations.}  The main step for deriving our algorithm for approximately computing PSD factorizations is to apply  the MM framework, with a meticulously chosen   auxilliary function,  to the convex quadratic SDP~\eqref{main}. Our main result is the following:

\begin{theorem}\label{LSstylethm}
Consider a fixed vector  $x\in \R^m_+$  and let
$\mathcal{A} : \mathbb{S}^{r} \rightarrow \mathbb{R}^{m}$ be the linear map defined by
$Z\mapsto \mathcal{A}(Z) = \left( \tr (A_1Z), \ \ldots \  ,\tr( A_mZ) \right),$
for some fixed $r \times r$ positive definite  matrices $A_1,\ldots, A_m$. Then, the objective function 
$\| x - \mathcal{A} (B) \|_2^2$ is non-increasing under the update rule
$$\bn = W(\mathcal{A}^\top x) W, \quad \text{where} \quad W= ([\mathcal{A}^{\sfT} \mathcal{A}] (\bo))^{-1}\#(\bo),$$
and moreover, if  initialized with a positive definite matrix, the iterates remain positive definite. 
\end{theorem}

\begin{proof}First, note that if the $A_i$'s and $\bo$ are all  positive definite, the  update rule is well-defined. Indeed, we have    $ [\mathcal{A}^{\sfT} \mathcal{A}] (\bo)=\sum_{k=1}^m\tr( A_k\bo) A_k$ is also positive definite, and thus invertible.

Set $F(B):=\| x - \mathcal{A} (B) \|_2^2$ and define the function
\begin{equation} \label{eq:auxilliary}
u_{\bo}(B) := F(\bo)+\la \nabla F(\bo), B-\bo\ra+\la B-\bo, T(B-\bo)),
\end{equation}
where $T: \mathbb{S}^r\to \mathbb{S}^r$ is the operator given  by
\begin{equation*}
T(Z) = W^{-1}ZW^{-1} \quad \text{ and } \quad W= ([\mathcal{A}^{\sfT} \mathcal{A}] (\bo))^{-1}\#(\bo).
\end{equation*}
The claim of the theorem will follow  as an immediate consequence of the MM framework, as long as we establish  that $u_{\bo}(B)$ is an auxilliary function, i.e., it satisfies the two properties given in~\eqref{auxilliary}. 

Clearly, we have that $u_{\bo}(\bo)=F(\bo)$, so it only remains to show the domination property, that is,  
 $u_{\bo}(B)\le F(B), $ for all $B\in \psd^r$.  In fact, we show a slightly stronger result, namely that  $u_{\bo}(B)\le F(B)$ holds  for all symmetric  matrices $B\in \mathbb{S}^r$.  To see this we use the second order Taylor expansion of $F$ at $\bo$, which  as $F$ is quadratic in $B$,  is given by
\begin{equation}\label{eq:taylor}
F(B)=F(\bo)+\la \nabla F(\bo), B-\bo\ra+\|\mca(B-\bo)\|_2^2.
\end{equation}
Comparing the expressions \eqref{eq:auxilliary} and \eqref{eq:taylor}, to show that $F(B)\le u_{\bo}( B)$ for all $B\in \mathbb{S}^r$ it suffices to check that the operator $T - \mathcal{A}^{\sfT} \mathcal{A}$ is positive; i.e., $\langle Z, [T - \mathcal{A}^{\sfT} \mathcal{A}] (Z) \rangle \ge 0$ for any matrix $Z\in \mathbb{S}^r$. 
This claim is the main technical part of the proof,  deferred to Lemma \ref{thm:domination_general} in  the Appendix.

Furthermore, the fact that $T - \mathcal{A}^{\sfT} \mathcal{A}$ is a positive operator, also implies that $T$ is itself   a positive operator. Consequently,  the MM update  \eqref{updaterule} obtained  by using   the auxilliary function \eqref{eq:auxilliary} can be calculated just by setting the gradient equal to zero, and is given by
\begin{equation} \label{eq:psd_update1}
B_{\mathrm{new}} = B_{\mathrm{old}} - T^{-1} \left( [\mathcal{A}^{\sfT} \mathcal{A}] (B_{\mathrm{old}}) - \mathcal{A}^{\sfT}(x) \right).
\end{equation}
Moreover, as $T^{-1} (Z) = W Z W$ and $W= ([\mathcal{A}^{\sfT} \mathcal{A}] (\bo))^{-1}\#(\bo)$ it follows that 
\be\label{cdsvs}
B_{\mathrm{old}} =W ([\mathcal{A}^{\sfT} \mathcal{A}] (B_{\mathrm{old}}))W= T^{-1} ( [\mathcal{A}^{\sfT} \mathcal{A}] (B_{\mathrm{old}})),
\ee
where for the first equality we use the unicity property of the matrix geometric mean (recall \eqref{riccati}). 
Subsequently, using \eqref{cdsvs}, the MM update rule in~\eqref{eq:psd_update1} simplifies to the following:
\be\label{updatenew}
B_{\mathrm{new}} = T^{-1}(\mathcal{A}^\sfT \bx)=W(\mathcal{A}^\sfT \bx)W.
\ee

Lastly, since the  $A_i$'s are PSD, it follows that $\cala^\sfT \bx=\sum_i x_i A_i$ is a conic combination of PSD matrices (recall that $x\in \R^m_+$), and thus, it is itself PSD.   Consequently,  $B_{\mathrm{new}}$ is PSD.   In fact, if the matrices $A_i$ and $B_{\mathrm{old}}$ are positive definite,  the updated matrix $B_{\mathrm{new}}$ is also  positive definite.
\end{proof}

%
%

Having established an iterative method  for problem \eqref{main} that is non-increasing in value and retains PSDness, we can  incorporate  this as a sub-routine  in our alternating optimization scheme  for computing  PSD factorizations. 
The pseudocode  of the resulting method is given in Algorithm~\ref{alg}.


%
\begin{algorithm}[h]
  \caption{ Matrix Multiplicative  algorithm for  computing PSD factorizations}
   \textbf{Input:} A matrix $X\in \R^{m\times n}_{\ge 0}$, parameter $r\in \mathbb{N}$\\
  \textbf{Output:}    $\{ A_1,\ldots, A_m \},\{ B_1, \ldots, B_n \} \subseteq \bbS^r_+$, $X_{ij}\approx \tr(A_iB_j)$ for all $i,j$
 \begin{algorithmic}\label{alg}
  \STATE while stopping criterion not satisfied: 
  \STATE 
  \be\label{updaterule}
  \bal 
A_i & \leftarrow V_i (\mcb^\top x_i) V_i \quad \text{ where } \quad V_i =([\mcb^{\sfT} \mcb] (A_i))^{-1}\# A_i, \quad x_i = X_{ i : }  \\
B_j & \leftarrow W_j (\mathcal{A}^\top x_j) W_j \quad \text{ where } \quad W_j = ([\mathcal{A}^{\sfT} \mathcal{A}] (B_j))^{-1}\#(B_j), \ x_j = X_{: j}
\eal 
\ee
\end{algorithmic}
\end{algorithm}

\section{Fixed Points of the MMU Algorithm}\label{sec:fp}

In this section, we show that the fixed points of the MMU algorithm satisfy the Karush-Kuhn-Tucker~(KKT)   optimality conditions for problem \eqref{bigboss}. 
 Letting $\{A_i^*\}_{i\in [m]},\{M^*_i\}_{i\in [m]}$ and $ \{B^*_j\}_{j\in [n]}, \{\Lambda^*_j\}_{j\in [n]}$ be   pairs of primal-dual optimal solutions of \eqref{bigboss} with zero duality gap, it is straightforward to verify that the KKT conditions  are 
$$
\begin{aligned}
& \tr(A^*_iM^*_i)=\tr(B^*_j\Lambda^*_j)=0, \quad  i\in [m], j\in [n] \\
&\mcb^\sfT (X_{: i})-[\mcb^\sfT\mcb] (A^*_i)=M^*_i, \quad i\in [m]\\
& \cala^\sfT (X_{: j}) -[\cala^\sfT \cala] (B^*_j)=\Lambda^*_j, \quad j \in [n].
\end{aligned}
$$
%
Furthermore, assuming that the primal optimal solutions $\{A^*_i\}_{i\in [m]}$ and  $\{B^*_j\}_{j\in [n]}$ are all  {\em positive definite}, it follows immediately from the complementary slackness conditions that $M^*_i=\Lambda^*_j=0$ for all $i\in [m], j\in [n]$. Consequently, in the special case of positive definite optimal solutions $\{A^*_i\}_{i\in [m]}$ and  $\{B^*_j\}_{j\in [n]}$ , the KKT conditions reduce to
\be\label{KKT}
\mcb^\sfT (X_{: i})=[\mcb^\sfT\mcb] (A^*_i), \ i\in [m] \quad \text{ and } \quad
 \cala^\sfT (X_{: j}) =[\cala^\sfT \cala] (B^*_j), \ j \in [n].
\ee
Based on the preceding discussion, in the next result (whose proof follows by Lemma \ref{thm:kkt_actualresult} in the  Appendix) shows that we can interpret our MMU algorithm as a fixed-point method for satisfying the KKT optimality conditions corresponding to  problem \eqref{bigboss}.
\begin{theorem}\label{fixedpointsthm}
If $\{A_i\}_{ i\in [m]}$  and $\{B_j\}_{j\in [n]}$  are positive definite fixed points of the update rule of the MWU algorithm given  in~\eqref{updaterule}, then they also satisfy the KKT conditions \eqref{KKT}.
\end{theorem}

\section{Applications of the MMU algorithm}\label{sec:applications}
\paragraph{Block-diagonal (BD) PSD factorizations.} If the  MMU algorithm is initialized with~BD  positive definite   matrices  with   {\em the same} block structure, the  BD structure is preserved at each update. 
Indeed, as $[\mathcal{A}^{\sfT} \mathcal{A}] (\bo)=\sum_{k=1}^m\tr( A_k\bo) A_k$ we see that $[\mathcal{A}^{\sfT} \mathcal{A}] (\bo)^{-1}$, and thus, $([\mathcal{A}^{\sfT} \mathcal{A}] (B_j))^{-1}\#(B_j)$ share the same block structure. Lastly, by definition of  the MMU algorithm~\eqref{updaterule}, $\bn$ is also  block-diagonal with the same structure. Thus, 
 if initialized with~BD-PSD matrices, the MMU algorithm  gives a method for computing a BD-PSD factorization.
%

\paragraph{Recovering  Lee-Seung's algorithm for NMF.} Diagonal matrices can be considered as  block-diagonal in a  trivial manner. Nevertheless, by the preceding discussion, if initialized with diagonal PSD matrices, the iterates of the MMU algorithm remain diagonal PSD throughout. In this special case, our MMU algorithm  
reduces to   Lee-Seung's~(LS) seminal Multiplicative Update algorithm  for   computing NMFs \cite{LS00}.  LS's algorithm is perhaps the most widely used  method for computing  NMFs    as it has succeeded to identify  meaningful features in  a diverse collection of real-life data sets and is extremely   simple  to implement. Specifically, LS's  updates are
\be\label{LSalg}
A\leftarrow  A \circ {XB^\top \over ABB^\top} \ \text{ and } \  B\leftarrow  B\circ {A^\top X\over A^\top A B},
\ee
where $X\circ Y, X/Y$ denote  the componentwise multiplication, division of two matrices respectively. Setting $A_i={\rm diag}(a_i)$ and $B_j={\rm diag}(b_j)$,  
the MMU algorithm updates $B_j$ as
$$
B_j \leftarrow B_j\Big(\sum_{i=1}^m \la a_i, b_j\ra A_i\Big)^{-1}\Big(\sum_{i=1}^mX_{ij}A_i\Big),
$$
 which is also a diagonal PSD matrix. Setting $A^\sfT =\begin{pmatrix} a_1^\sfT & \ldots & a_m^\sfT \end{pmatrix}$ and $ B=\begin{pmatrix}b_1 & \ldots & b_n \end{pmatrix}$, we immediately recover  LS's  update rule~\eqref{LSalg}. 

\paragraph{PSD factorizations for nonnegative tensors.}  Motivated by PSD factorizations of nonnegative matrices,  \cite{tensor} define  an $r$-dimensional  PSD factorization of a nonnegative tensor $T$ (with $n$ indices of dimension $d$) as a collection of PSD matrices $C^{(1)}_{i_1}, \ldots, C^{(n)}_{i_n} \in \psd^r $ for all $i_k\in [d]$  such that
  \be\label{tensor}
     T_{i_1, \ldots, i_n} = \sum_{i, j=1}^r C^{(1)}_{i_1}(i,j) \cdots C^{(n)}_{i_n}(i,j), \  \text{ for all } i_k \in [d].
     \ee
Equivalently, and more succinctly, we may write \eqref{tensor} as 
$$   T_{i_1\ldots i_n}={\rm sum}(C^{(1)}_{i_1}\circ \cdots \circ C^{(n)}_{i_n}),  \text{ for all }  i_k \in [d], $$ where $\circ $ denotes the Schur product of matrices and ${\rm sum}(X)=\sum_{ij}X_{ij}$. The motivation for studying tensor PSD factorizations  comes from the fact that they characterize   the quantum correlation complexity for generating multipartite classical distributions \cite{tensor}. 
  We now show how our MMU algorithm can be used as a primitive to calculate tensor  PSD factorizations.  
   For simplicity of presentation we restrict  to  $n=3$     and~consider
    \be\label{vsvs}
    \inf  \sum_{i_1, i_2, i_3}\left( T_{i_1 i_2i_3} -{\rm sum}(C^{(1)}_{i_1}\circ C^{(2)}_{i_2} \circ C^{(3)}_{i_3})\right)^2  \quad  \text{subject to} \quad  C^{(1)}_{i_1}, C^{(2)}_{i_2},  C^{(3)}_{i_3} \in \psd^r.
    \ee
   As in the case of PSD factorizations, to solve \eqref{vsvs} we employ a block coordinate descent approach. Specifically,  fixing all matrices $C^{(1)}_{i_1}, C^{(2)}_{i_2}$ the optimization problem \eqref{vsvs} is separable wrt each $C^{(3)}_{i_3}$ for all $i_3\in~[d].$ 
  Thus,   we need to solve the optimization problem
    \be\label{vsf}
    \arg \inf_C \sum_{i_1, i_2}\left(T_{i_1 i_2i_3} - \langle C^{(1)}_{i_1}\circ C^{(2)}_{i_2}, C\rangle\right)^2 \quad  \text{ subject to } \quad C\in \psd^r.
    \ee
Defining the  map  $\mca: \mathbb{S}^r \to \mathbb{R}^{d^2}$ where $  X\mapsto (\la X, C^{(1)}_{i_1}\circ C^{(2)}_{i_2}\ra)_{i_1,i_2} $, problem \eqref{vsf} is 
equivalent to
\be\label{csvsdv}
\arg \inf \|{\rm vec}(T_{::i_3})-\mca(C)\|^2_2 \quad \text{ subject to} \quad C\in \psd^r,
\ee
for all $i_3\in [d]$. Note that $\mca^\sfT: \mathbb{R}^{d^2} \to  \mathbb{S}^r $ where $   x=(x_{x_1x_2})\mapsto \sum_{i_1, i_2\in [d]}x_{i_1i_2}C^{(1)}_{i_1}\circ C^{(2)}_{i_2} $ is a PSD matrix, as the Schur product of PSD matrices is PSD. Thus,  Theorem \ref{LSstylethm} gives  an update rule that preserves PSDness and  for which the objective function \eqref{csvsdv}  is nonincreasing.

\section{Numerical experiments}\label{sec:numerical}

\paragraph{Distance matrices.}  
Let $v \in \mathbb{R}^{n}$ be a vector and let $M$ be a $n\times n$ matrix whose entries are $M_{ij} = (v_i-v_j)^2$.  $M$ is known as a \emph{distance matrix} and a 2-dimensional PSD factorization is  $M_{ij} = \mathrm{tr}(A_i B_j)$, where
$A_i = \left(\begin{smallmatrix}1 \\ v_i \end{smallmatrix}\right) \left(\begin{smallmatrix} 1 \\ v_i \end{smallmatrix}\right)^\sfT $ and $ 
B_j = \left(\begin{smallmatrix} -v_j \\ 1 \end{smallmatrix}\right) \left(\begin{smallmatrix} -v_j \\ 1 \end{smallmatrix}\right)^\sfT. $
We generate a random  $v \in \mathbb{R}^{n}$ with $n=20$ where each entry is drawn from the standard normal distribution.  We apply our algorithm to compute a 2-dimensional  factorization whereby we perform $500$ iterations over $50$ random initializations.  For conditioning reasons, we perform a damped update whereby we add $\epsilon I$, $\epsilon = 10^-8$, to each matrix variable at every iterate.  We compute the normalized squared error loss of the factorization from the data matrix, and we plot the error over each iteration $\mathrm{Err} =  \sum_{i,j} (\mathrm{tr}( A_i B_j ) - M_{ij} )^2 /\sum_{i,j} M_{ij}^2$  in Figure \ref{fig:distancematrx_errors}.
Our experiments  suggest that, with sufficient random initializations, our algorithm finds a PSD factorization that is close to being exact. 

\begin{figure}[h!]
\centering
\includegraphics[width=0.4\textwidth]{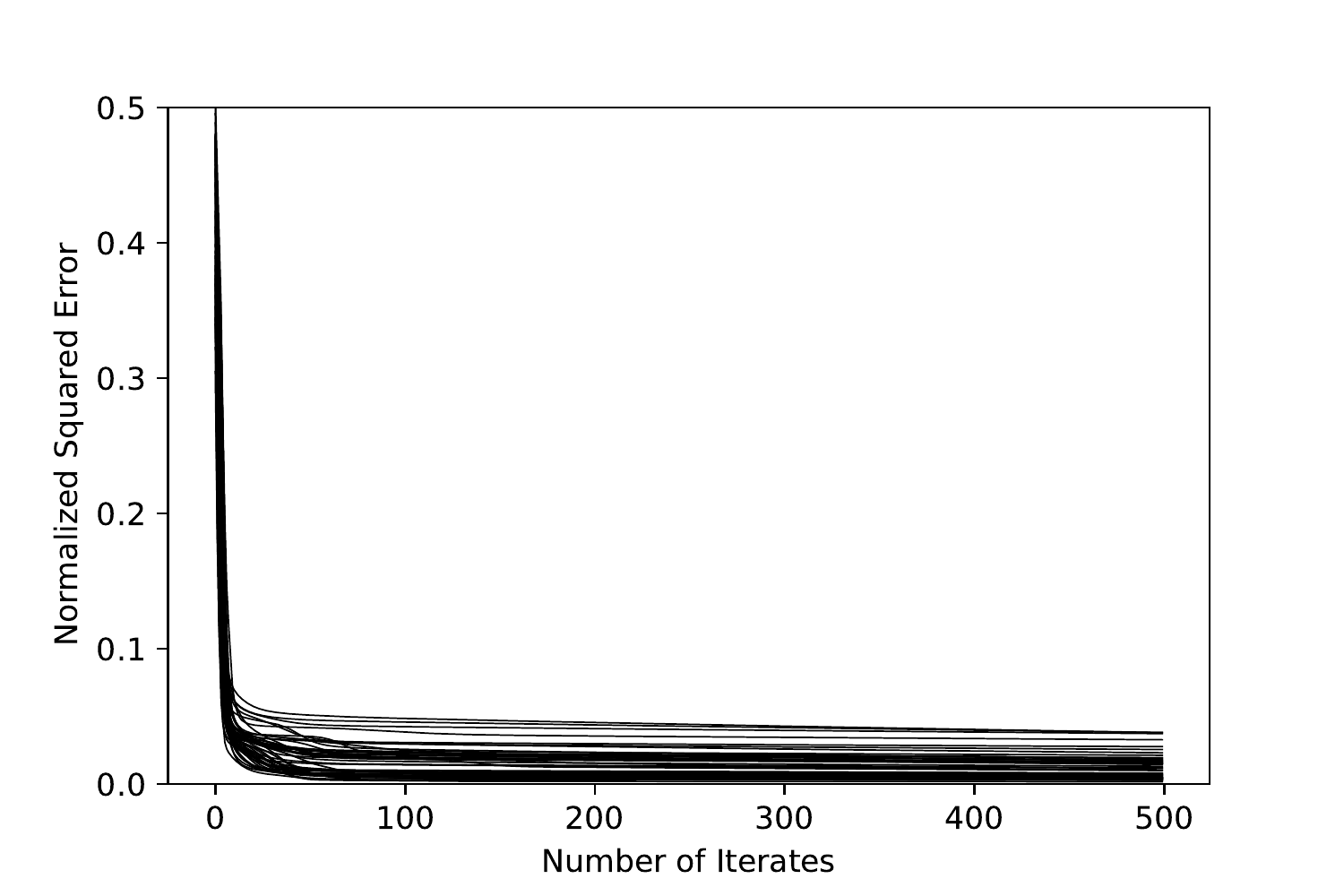}
\includegraphics[width=0.4\textwidth]{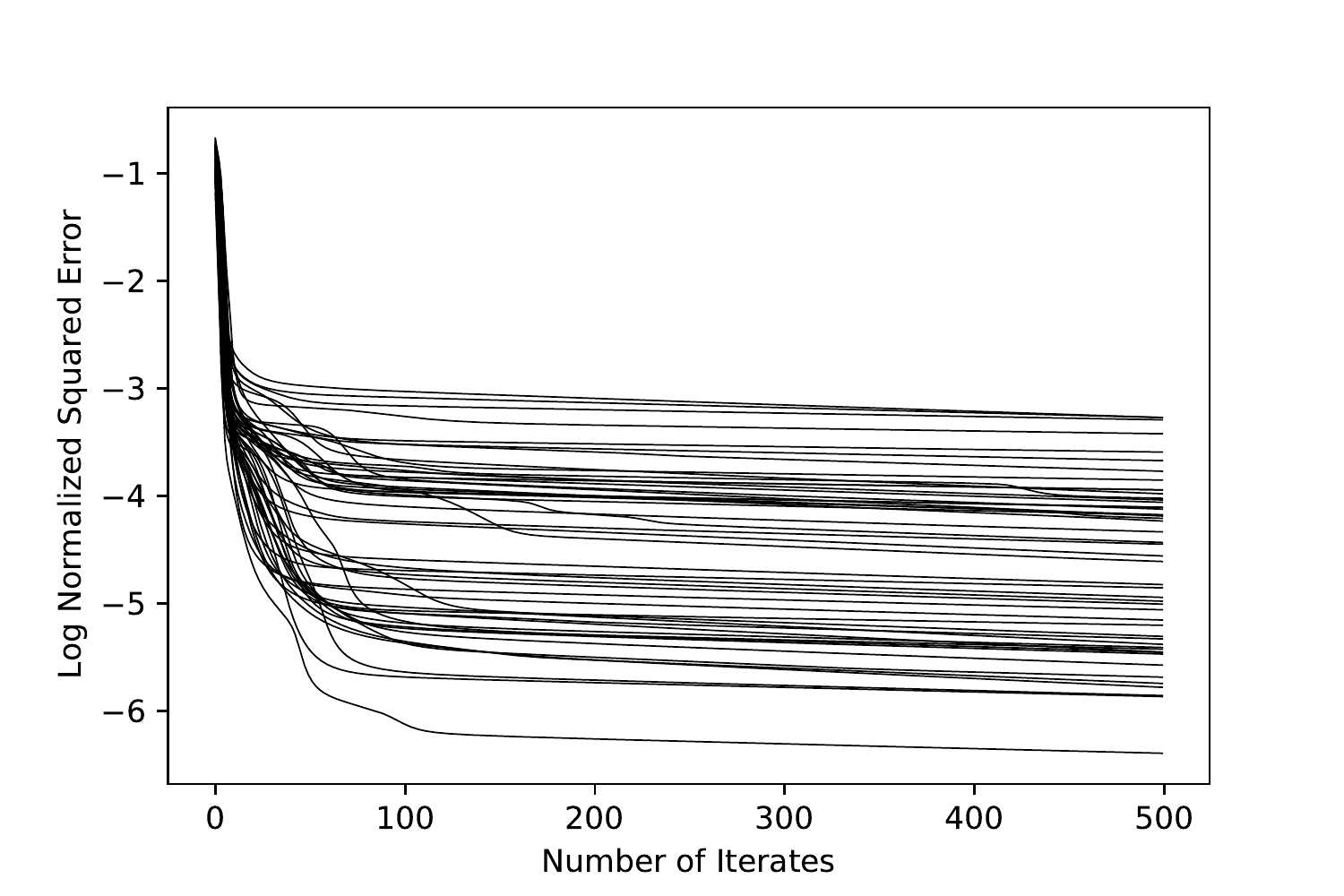}
\caption{Performance of the MWU algorithm for computing a PSD factorization of a distance matrix.  Different curves correspond to different random initializations. }
\label{fig:distancematrx_errors}
\end{figure}


\paragraph{CBCL Face Image Dataset.}  In our second  experiment we apply  the MMU  method to compute a PSD factorization  of face image data from the CBCL Face Database \cite{CBCL}. 
The objective of this experiment is to explain how computing a PSD factorization can be viewed as a representation learning algorithm that generalizes NMF.  
 The CBCL dataset comprises $2429$ images of faces, each of size $19 \times 19$ pixels,  linearly scaled so that the pixel intensity has mean $0.5$ and standard deviation $0.25$, with values  subsequently clipped at $[0,1]$.  The resulting data matrix has size $361 \times 2429$.   Our experiments were conducted in Python on an Intel 7-th Gen i7 processor at 2.8GHz.

Recall that    an $r$-dimensional PSD factorization   of $X$ given by   $\{A_i\}, \{B_j\}$ 
 gives rise   to a representation  
$X_{:j}=\sum_i\lambda_i\cala(u_iu_i^\sfT)$, where  $B_j=\sum_{i=1}^r\lambda_iu_iu_i^\sfT$ is the  spectral decomposition  
of~$B_j$ and $\cala: \mathbb{S}^r \to \R^m$ is the   linear mapping satisfying $Z\mapsto (\tr(A_1Z), \ldots, \tr(A_mZ)).$
In all figures in this section, given a specific image $X_{:j}$ in the CBCL dataset, we illustrate the images corresponding to the building  blocks identified by the MMU algorithm, namely $\cala(u_iu_i^\sfT)$ for all eigenvectors of $B_j$.

 As our baseline, we compute a 27-dimensional  NMF of the CBCL data matrix  over $500$ iterations. In Figure \ref{fig:1x1_full} below 
 we illustrate  the decomposition of one of  images from  the dataset  into these $27$ constituents. 
 
\begin{figure}[h]
\centering
\includegraphics[width=1\textwidth]{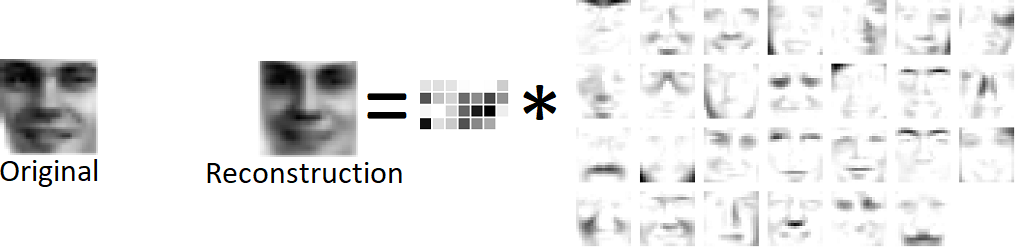}
\caption{Image decomposition into  building blocks learned from 27-dimensional NMF.}
\label{fig:1x1_full}
\end{figure} 
Next, we calculate a 7-dimensional PSD factorization of the CBCL data matrix over 500 iterations and 
illustrate the decomposition of the same image in the new basis in Figure \ref{fig:PSD_full}. 
   We compute a factorization of size $7$ to match the number of degrees of freedom in Figure \ref{fig:1x1_full}.
We note that the constituents learned from the PSD factorization  are less interpretable than those learned using~NMF -- here, we use the word `interpretable' in a loose sense to mean that some factors learned from NMF are easily interpreted to describe local facial features such as the eyes or the nose whereas the constituents learned from PSD are of a global nature. Nevertheless, we note that this  phenomenon that has been  been also observed for NMF applied to  datasets beyond  CBCL \cite{nonatomicNMF}. 
\begin{figure}[h]
\centering
\includegraphics[width=1\textwidth]{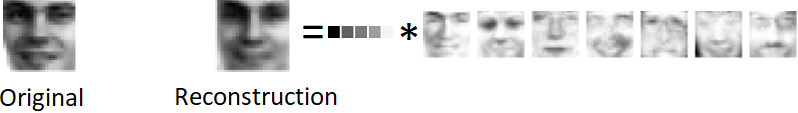}
\caption{Image decomposition  into  building blocks learned from 7-dimensional PSD factorization.}
\label{fig:PSD_full}
\end{figure}

%
%
%
%

 Lastly, we use the MMU algorithm to 
 compute a block-diagonal PSD factorization with  $9$ blocks, each  of size $2$, over $500$ iterations (the number of degrees of freedom is comparable), which we illustrate the decomposition in Figure \ref{fig:2x2_full}.  In this instance, the constituents contain more localized~features.  
 
 \begin{figure}[h]
\centering
\includegraphics[width=1\textwidth]{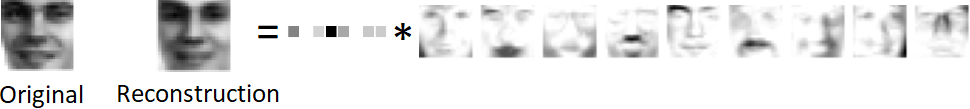}
\caption{Image decomposition  into  building blocks learned using  $2\times 2$-block PSD factorization.}
\label{fig:2x2_full}
\end{figure}

 An advantage of learning a continuum of basic building blocks is that one can express certain geometries in the data that is otherwise not possible using a finite number of building blocks -- as an illustration of this intuition, in Figure \ref{fig:2x2_atoms}, we show a continuum of atoms corresponding to a single $2\times 2$ block, which  capture a continuum between the nose and the nostrils.

\begin{figure}[h!]
\centering
\includegraphics[width=1\textwidth]{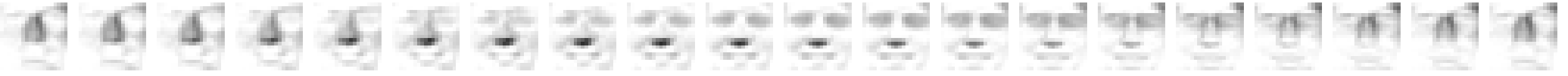}
\caption{Visualization of continuum of  building blocks learned using $2\times 2$-block PSD factorization.}
\label{fig:2x2_atoms}
\end{figure}

\paragraph{Acknowledgments.}  YS gratefully acknowledges Ministry of Education (Singapore) Academic Research Fund (Tier 1) R-146-000-329-133.  AV gratefully acknowledges Ministry of Education (Singapore) Start-Up Research Grant SRG ESD 2020 154 and NRF2019-NRF-ANR095 ALIAS grant. 

\bibliography{PSD_MM,NMF}

\appendix

\section{Proof of the domination property}\label{algorithmic}
In this section we prove  the  domination property, which is the last remaining ingredient in the proof of Theorem \ref{LSstylethm}.

\begin{lemma} \label{thm:domination_general}
Fix $r\times r$ positive definite matrices $A_1,\ldots,A_m$ and consider the linear mapping 
	 	\begin{equation*}
	\mathcal{A} : \mathbb{S}^{r} \rightarrow \mathbb{R}^{m} \text{ where } Z\mapsto \left(
	\langle A_1,Z \rangle, \ \ldots \ 
	,\langle A_m,Z \rangle
	 \right).
\end{equation*}
 Moreover, fix an $r\times r$ positive definite matrix   $B$  and set
			$$W = [\mathcal{A}^{\sfT} \mathcal{A}] (B)\#B^{-1}.$$ 
Then, we have that
	\begin{equation}\label{eq:main}
		\langle Z, WZW \rangle \geq \langle Z, [\mathcal{A}^{\sfT} \mathcal{A}] (Z) \rangle,  \text{ for all } Z \in \mathbb{S}^{r}.
	\end{equation}
\end{lemma}
\begin{proof}
First note that $W$ is well-defined as  $[\mathcal{A}^{\sfT} \mathcal{A}] (B)=\sum_{k=1}^m\la A_k,B\ra A_k$ is positive definite. First,  we  reduce the theorem to the special case  where $B=I$.
For this, define the linear map 
$$\tilde{\mathcal{A}} : \mathbb{S}^{r} \rightarrow \mathbb{R}^{m}, \quad  Z \mapsto  \mathcal{A}(B^{1/2}ZB^{1/2}).$$
  Noting that 
$$\tilde{\mathcal{A}}(I)= \cala(B) \ \text{ and } \ \tilde{\mathcal{A}}^\sfT(Z)=B^{1/2}\cala^\sfT(Z)B^{1/2},$$
it follows that 
	\begin{equation*}
		\begin{aligned}
			W & = [\mathcal{A}^{\sfT} \mathcal{A}] (B)\#B^{-1}=B^{-1}\# [\mathcal{A}^{\sfT} \mathcal{A}] (B) \\
			  & = B^{-1/2} (B^{1/2} [\mathcal{A}^{\sfT} \mathcal{A}] (B) B^{1/2})^{1/2} B^{-1/2} \\
			  & = B^{-1/2} ([\tilde{\mathcal{A}}^{\sfT} \tilde{\mathcal{A}}] (I) )^{1/2} B^{-1/2}.
		\end{aligned}
	\end{equation*}
	Furthermore, setting 
$$\tilde{Z} = B^{-1/2} Z B^{-1/2},$$ we get that 
	\begin{equation*}
		\begin{aligned}
			\langle Z, WZW \rangle ~=~ & \tr(ZWZW) \\
			~=~ & \tr(ZB^{-1/2} ([\tilde{\mathcal{A}}^{\sfT} \tilde{\mathcal{A}}] (I) )^{1/2} B^{-1/2}ZB^{-1/2} ([\tilde{\mathcal{A}}^{\sfT} \tilde{\mathcal{A}}] (I) )^{1/2} B^{-1/2}) \\
			~=~ & \tr(B^{-1/2}ZB^{-1/2} ([\tilde{\mathcal{A}}^{\sfT} \tilde{\mathcal{A}}] (I) )^{1/2} B^{-1/2}ZB^{-1/2} ([\tilde{\mathcal{A}}^{\sfT} \tilde{\mathcal{A}}] (I) )^{1/2} ) \\
			~=~ & \langle \tilde{Z}, ([\tilde{\mathcal{A}}^{\sfT} \tilde{\mathcal{A}}] (I) )^{1/2} \tilde{Z} ([\tilde{\mathcal{A}}^{\sfT} \tilde{\mathcal{A}}] (I) )^{1/2} \rangle. 
		\end{aligned}
	\end{equation*}
	Similarly
	\begin{equation*}
		\begin{aligned}
			\langle Z, [\mathcal{A}^{\sfT} \mathcal{A}] (Z) \rangle & = \langle B^{1/2} B^{-1/2} Z B^{-1/2} B^{1/2}, [\mathcal{A}^{\sfT} \mathcal{A}] (B^{1/2} B^{-1/2} Z B^{-1/2} B^{1/2}) \rangle \\
			& = \langle \tilde{Z} , B^{1/2}[\mathcal{A}^{\sfT} \mathcal{A}] (B^{1/2} \tilde{Z} B^{1/2}) B^{1/2} \rangle \\
			& = \langle \tilde{Z} , [\tilde{\mathcal{A}}^{\sfT} \tilde{\mathcal{A}}] (\tilde{Z}) \rangle.
		\end{aligned}
	\end{equation*}

To conclude the proof it remains to prove the theorem statement in the special case where $B=I$. Note that when $B=I$  we have 
$$W = [\mathcal{A}^{\sfT} \mathcal{A}] (I)\#I^{-1}=[\mathcal{A}^{\sfT} \mathcal{A}] (I)\#I=I\#[\mathcal{A}^{\sfT} \mathcal{A}] (I)=([\mathcal{A}^{\sfT} \mathcal{A}] (I))^{1/2},$$
where in the second to last equality we used that the geometric mean is symmetric in its two arguments and the least equality follows by the definition of the geometric mean \eqref{GM}.

Furthermore, by the definition  of the map $\cala$ we have that 
$$[\cala^\sfT\cala](Z)=\sum_{k=1}^m\la A_k,Z\ra A_k,$$
which in particular implies that  
	$$W = ([\mathcal{A}^{\sfT} \mathcal{A}] (I))^{1/2} = \lp\sum_{k} \tr(A_k)  A_k\rp^{1/2}.$$
	Thus, in the case where $B=I$, the inequality we need to prove (cf. Equation \eqref{eq:main}) specializes to: 
			
			\be\label{main:inequality}
			\tr\lp Z \lp\sum_{k} \tr(A_k)  A_k\rp^{1/2} Z \lp\sum_{k} \tr(A_k)  A_k\rp^{1/2}\rp   \geq \sum_k \tr(A_kZ)^2,  \text{ for all } Z \in \mathbb{S}^{r}.
 \ee
As a special case of  Lieb's concavity theorem (e.g. see \cite[Theorem 6.1]{carlen}) we have that for any fixed real symmetric matrix $Z\in \bbS^n$, the real-valued map
$$(X,Y)\mapsto \tr(ZX^{1/2}ZY^{1/2})$$ 
is concave over  the domain  $\bbS^n_{++}\times \bbS^n_{++}$. In particular,  for any 
 $(X_1, Y_1), (X_2, Y_2)\in \bbS^n_{++}\times \bbS^n_{++}$ and $\lam \in [0,1]$ we have that  
 $$\tr(Z(\lam X_1+(1-\lam)X_2)^{1/2}Z(\lam Y_1+(1-\lam)Y_2)^{1/2})\ge \lam \tr(ZX_1^{1/2}ZY_1^{1/2})+  (1-\lam) \tr(ZX_2^{1/2}ZY_2^{1/2}).$$
 In turn, setting $X_1=Y_1={\tr(A_1)\over \lam}A_1$ and $X_2=Y_2={\tr(A_2)\over 1-\lam}A_2$  we get that:
  $$
  \begin{aligned}
  &\tr(Z(\tr(A_1)A_1+\tr(A_2)A_2)^{1/2}Z(\tr(A_1) A_1+\tr(A_2)A_2)^{1/2})\ge \\
  &\tr(A_1) \tr(ZA_1^{1/2}ZA_1^{1/2})+   \tr(A_2)\tr(ZA_2^{1/2}ZA_2^{1/2}). 
  \end{aligned}
  $$
  Finally, by induction it follows that:
   \begin{equation}\label{sdvsvbs}
   \tr\lp Z(\sum_k \tr(A_k)A_k)^{1/2}Z( \sum_k (\tr(A_k)A_k)^{1/2}\rp \ge \sum_k \tr(A_k) \tr(ZA_k^{1/2}ZA_k^{1/2}),
   \end{equation}
  for any family of positive definite matrices $A_1, \ldots, A_m$ and any fixed matrix $Z$. 
    Based on \eqref{sdvsvbs}, to prove   \eqref{main:inequality} it remains to show that
    
    \be\label{cdvd}    \sum_k \tr(A_k) \tr(ZA_k^{1/2}ZA_k^{1/2})\ge \sum_k \tr(A_kZ)^2. 
    \ee
    We prove  inequality \eqref{cdvd} term-by-term, i.e., we show that 
        \be\label{cdsvds}
         \tr(A_k) \tr(ZA_k^{1/2}ZA_k^{1/2})=\tr(A_k)\tr((ZA_k^{1/2})^2)\ge  \tr(A_kZ)^2. 
         \ee
        	For this we use that for any  $X,Y \in \mathbb{S}^{n}$ we have 
	\begin{equation}\label{csvsd}
		\tr(XY)^2
		 \leq \tr(X^2) \cdot \tr(Y^2). 
	\end{equation}
Indeed,  let  $\sigma_1 \leq \ldots \leq \sigma_n$ be the singular values of $X$ and $\mu_1 \leq \ldots \leq \mu_n$ be the singular values of $Y$.  Then
\begin{equation*}
	\tr(XY)^2 \leq \left( \sum_{i=1}^{ n} \sigma_i \mu_i \right)^2 \leq \left( \sum_{i=1}^{n} \sigma_i^2 \right) \left( \sum_{i=1}^{n} \mu_i^2 \right) = \tr(X^2) \cdot \tr(Y^2).
\end{equation*}
The first inequality follows from the von Neumann's trace inequality (see e.g. \cite{bhatia}) and the second inequality follows from the Cauchy-Schwarz inequality.

Lastly, by \eqref{csvsd} we have that 
$$\tr(A_kZ)^2=\tr(A_k^{1/2}(A_k^{1/2}Z))\le \tr(A_k)\tr((ZA_k^{1/2})^2),$$
which is exactly 
\eqref{cdsvds}.
\end{proof}

\section{Fixed-points of the MWU algorithm are KKT points}
In this section we give the proof of Theorem \ref{fixedpointsthm}. 

\begin{lemma} \label{thm:kkt_actualresult} Let $\mathcal{A}$ be the linear map $\mathcal{A}(Z) = (\mathrm{tr}(A_1 Z), \ldots, \mathrm{tr}(A_m Z))^\sfT$, where $A_i$ are positive definite matrices, and let $x$ be a vector with positive entries.  Suppose $B$ is a positive definite matrix satisfying $B =W (\mathcal{A}^\top x) W$, where $W = ([\mathcal{A}^{\sfT} \mathcal{A}] (B))^{-1}\#B$.  Then $\cala^\sfT x =[\cala^\sfT \cala] (B)$.
\end{lemma}

\begin{proof}  Since $B$ is positive definite, it follows that $[\mathcal{A}^{\sfT} \mathcal{A}] (B)=\sum_{i=1}^m\la A_i,B\ra A_i$ is also positive definite, and thus $W = ([\mathcal{A}^{\sfT} \mathcal{A}] (B))^{-1}\#B$ is positive definite.  We then have that
$$\mathcal{A}^\top x=W^{-1}BW^{-1}=(B^{-1}\#[\mathcal{A}^{\sfT} \mathcal{A}] (B))B(B^{-1}\#[\mathcal{A}^{\sfT} \mathcal{A}](B))=[\mathcal{A}^{\sfT} \mathcal{A}](B),$$
where 
the last equality follows by the unicity property of the matrix geometric mean \eqref{riccati}.
\end{proof}

\end{document}